\documentclass[a4paper,intlimits]{article}
\usepackage{latexsym,amsfonts,amssymb,amsmath,amsthm}
\date{}

\newtheorem{theorem}{Theorem}[section]
\newtheorem{lemma}[theorem]{Lemma}

\newtheorem{proposition}[theorem]{Proposition}

\theoremstyle{definition}
\newtheorem{definition}[theorem]{Definition}

\theoremstyle{remark}

\newcommand{\NN}{\ensuremath{\mathbb{N}}}
\newcommand{\RR}{\ensuremath{\mathbb{R}}}

\newcommand{\abs}[1]{\ensuremath{\lvert#1\rvert}}
\newcommand{\norm}[1]{\ensuremath{\lVert#1\rVert}}

\DeclareMathOperator{\card}{card}

\DeclareMathOperator{\conv}{conv}
\DeclareMathOperator{\Fix}{Fix}

\DeclareMathOperator{\spanned}{span}

\begin{document}

\title{The $C^1$ property of convex carrying simplices for three-dimensional competitive maps\footnote{This is an Accepted Manuscript of an article published by Taylor \& Francis in Journal of Difference Equations and Applications on 31/01/2018, available online: https://www.tandfonline.com/doi/full/10.1080/10236198.2018.1428964.}}

\author{
Janusz Mierczy\'nski \\
Faculty of Pure and Applied Mathematics \\
Wroc{\l}aw University of Science and Technology \\
Wybrze\.ze Wyspia\'nskiego 27 \\
PL-50-370 Wroc{\l}aw \\
Poland}

\maketitle

\begin{abstract}
It is proved that a convex carrying simplex for a three-dimensional competitive map is a $C^1$ submanifold-with-corners neatly embedded in the non-negative octant.
\end{abstract}



\section{Introduction}
\label{sect:intro}
In his paper \cite{H-1988} M. W. Hirsch proved, for a wide class of strongly competitive systems of ordinary differential equations (ODEs), the existence of an unordered (with respect to the coordinate-wise ordering) invariant set, homeomorphic to the standard probability simplex, such that every non-trivial (that is, not equal identically to zero) orbit is attracted towards it.  M. L. Zeeman in~\cite{Z-1993} introduced the name `carrying simplex'.

As regards discrete-time competitive dynamical systems, the existence of carrying simplices was proved first for diffeomorphisms, see~\cite{Sm, W-J-2002}.  For competitive maps that are not necessarily invertible, see, e.g., 
\cite{H-2008, D-W-Y, RH, J-N-W, B-JDDE, J-N}.  One should also mention here earlier papers on so-called d-hypersurfaces, see \cite{Tak-1990, Tak-1992, W-J-2002}.

\medskip
To the best of our knowledge, it was first noticed in \cite{Z} that for a two-dimensional Lotka--Volterra system of ODEs possessing an equilibrium in the interior of the first quadrant the carrying simplex is strictly convex if and only if such an equilibrium (necessarily unique) is attracting.  Some errors in \cite{Z} were later fixed in~\cite{T, T2}, see also a completely independent paper~\cite{F-T-dS}.  In the case of Lotka--Volterra ODE systems of any dimension, \cite[Cor.~4.5]{Z-Z} states, for a system having a unique equilibrium in the interior of the non-negative orthant, that if the carrying simplex is convex then that equilibrium is globally asymptotically stable.  For other results on convexity, see~\cite{Z-Z-2002}.

The interest in investigating convexity and concavity of the carrying simplex has been rekindled in the series of papers \cite{B-Edinburgh, B} (for Lotka--Volterra systems of ODEs) and in~\cite{B-JDDE, B-H-2017} (for some discrete-time systems). In~particular, for two-dimensional Leslie--Gower models for which the origin is a repeller the global asymptotic stability of a (unique) fixed point in the interior of the first quadrant implies that the carrying simplex is convex, see~\cite[Sect.~5]{B-JDDE}.

\smallskip
In the light of the above, there is strong correlation between the convexity of the carrying simplex and the global asymptotic stability of the unique equilibrium/fixed point in the interior of the orthant, at~least for Lotka--Volterra/Leslie--Gower systems.

\medskip
Another feature of the carrying simplex is its smoothness. It was proved first for totally competitive systems of ODEs in~\cite{M-1994} that if a system is weakly persistent (meaning that no orbit is fully attracted towards the boundary of the orthant) then the carrying simplex is a neatly embedded $C^1$ submanifold-with-corners.  In~\cite{J-M-W} a characterization of the neat embedding for competitive maps was given in terms of inequalities between the Lyapunov exponents for ergodic invariant measures supported on the boundary of the carrying simplex.  For other results, see \cite{Brunovsky, Benaim, M-1999, M-1999a, M-1999b}.

\bigskip
As is shown in the present paper, for a wide class of three-dimensional competitive maps, not limited to Lotka--Volterra systems, the convexity of the carrying simplex alone entails that the simplex is a
$C^1$ submanifold-with-corners neatly embedded in the first octant.

The proofs use results from~\cite{J-M-W}.  As mentioned earlier, it is proved there that the neat embedding of the carrying simplex is equivalent to the fulfillment of some inequalities between Lyapunov exponents of ergodic invariant measures supported on the boundary.  In our case the ergodic invariant measures on the boundary of the carrying simplex are just Dirac measures on fixed points, and their Lyapunov exponents are the logarithms of (the moduli of) the eigenvalues of the Jacobian matrices at those boundary fixed points.  Even now it is not trivial to show that the convexity of the carrying simplex implies the appropriate inequalities.

\smallskip
The paper is organized as follows.  In Section~\ref{sect:preliminaries} the main notations are given, and the standing assumptions are formulated.  Section~\ref{sect:main} is devoted to the statement and proof of the main result (Theorem~\ref{thm:main}).  Section~\ref{sect:conclusions} contains concluding remarks.

\section{Preliminaries}
\label{sect:preliminaries}

In this paper we shall distinguish between points (elements of the {\em affine\/} space $H = \{\, x = (x_1, x_2, x_3): x_i \in \RR \,\}$) and vectors (elements of the {\em vector\/} space $V = \{\, v = (v_1, v_2, v_3): v_i \in \RR \,\}$).  $\norm{\cdot}$ stands for the Euclidean norm in $V$.

Denote by $C = \{\,x \in H: x_i \ge 0$ for all $i = 1, 2, 3 \,\}$ the three-dimensional non-negative {\em octant\/}.

\smallskip
Given $\emptyset \neq I \subset \{1,2,3\}$, we write $\overline{I} := \{1, 2, 3 \} \setminus I$.

Let $H_I := \{\,x\in H: x_j = 0$ for $j \in \overline{I} \,\}$.  For two points $x, y \in H_I$, we write $x \leq_I y$ if $x_i \leq y_i$ for all $i \in I$, and $x \ll_I y$ if $x_i < y_i$ for all $i \in I$.  If $x \leq_I y$ but $x \neq y$ we write $x <_I y$ (the subscript in $\leq$, $<$, $\ll$ is dropped if $I = \{1,2,3\}$).

The interior of $C$ is $C^{\circ} := \{\,x \in H : x \gg 0\,\}$ and the boundary of $C$ is $\partial C = C \setminus C^{\circ}$. We also let $H_I^+ := C \cap H_I$, $\dot{H}^+_I := \{\,x \in H_I^+ : x_i > 0$ for $i \in I \,\}$ and $\partial H_I^+$ be the relative boundary of $H_I^+$, $\partial H_I^+ = H_I^+ \setminus \dot{H}^+_I$.  $H^{+}_I$ is called a {\em $k$-dimensional face\/} of $C$, where $k = \card{I}$.

Let $V_I := \{\,v \in V: v_j = 0$ for $j \in \overline{I} \,\}$.  For two vectors $v, v \in V_I$, we write $v \leq_I w$ if $v_i \leq w_i$ for all $i \in I$, and $v \ll_I w$ if $v_i < w_i$ for all $i \in I$.  If $v \leq_I w$ but $v \neq w$ we write $v <_I w$ (the subscript in $\leq$, $<$, $\ll$ is dropped if $I = \{1,2,3\}$).

The standard non-negative {\em cone\/} $K$, with nonempty interior $K^{\circ}$, in $V$ is the set of all $v$ in $V$ such that $v_i \ge 0$ for all $i \in \{1, 2, 3\}$.  For $I \subset \{1, 2, 3\}$, we write $K_I := K \cap V_I$, $\dot{K}_I := \{\,v \in K_I: v_i > 0$ for $i \in I\,\}$, $\partial K_{I} := K_I \setminus \dot{K}_I$.

\medskip
Let $P \colon C \to P(C)$ be a $C^k$ diffeomorphism onto its image $P(C) \subset H$.  Recall that this means that there is an open $U \subset H$, $C \subset U$, and a $C^k$ diffeomorphism $\tilde{P} \colon U \to \tilde{P}(C)$ such that the restriction $\tilde{P}|_{C}$ of $\tilde{P}$ to $C$ equals $P$.

\smallskip
A set $A \subset C$ is {\em invariant\/} if $P(A) = A$.  For $x \in C$ the {\em $\omega$-limit set $\omega(x)$\/} of $x$ is the set of those $y \in C$ for which there exists a subsequence $n_k \to \infty$ such that $\norm{P^{n_k}x - y} \to 0$ as $k \to \infty$.  A compact invariant $\Gamma \subset C$ is the {\em global attractor\/} for $P$ if for each bounded $B \subset C$ and each $\epsilon > 0$ there exists $n_0 \in \NN$ such that $P^{n}(B)$ is contained in the $\epsilon$-neighbourhood of $\Gamma$
for $n \ge n_0$.

\smallskip
The symbol $\triangle$ stands for the {\em standard probability
$2$-simplex}, $\triangle : =\{\,x \in C: \sum_{i = 1}^3 x_i = 1\,\}$.  We
write $\triangle_I := \triangle \cap H_I^+$, $\dot{\triangle}_I :=
\triangle \cap \dot{H}_I^+$, $\partial\triangle_I := \triangle \cap
\partial H_I^+$.

\medskip
As in~\cite{J-M-W} we introduce the following assumptions.  We assume throughout the paper that they are satisfied.

\vskip 2mm
\noindent\textbf{(H1)} {\em $P$ is a $C^2$ diffeomorphism onto its
image $P(C)$.}

\vskip 2mm
\noindent \textbf{(H2)} {\em For each nonempty $I \subset \{1,2,3\}$, the sets $A = H^{+}_I$, $\dot{H}^+_I$ and $\partial H^+_I$ have the property that $P(A) \subset A$ and $P^{-1}(A) \subset A$.}

\vskip 2mm
\noindent \textbf{(H3$'$)} {\em For each nonempty $I \subset \{1,2,3\}$ and $x \in \dot{H}^+_I$, the $I \times I$  Jacobian matrix $D(P|_{H_I^+})(x)^{-1} = (DP(x)^{-1})_I = (DP^{-1}(Px))_I$ has all entries positive. Moreover, for any $v \in K_{\overline{I}} \setminus \{0\}$ there exists some $j \in I$ such that $(DP(x)^{-1}v)_j > 0$}

\vskip 2mm
\noindent \textbf{(H4$'$)} {\em For each $i \in \{1,2,3\}$, $P|_{H^+_{\{i\}}}$ has a unique fixed point $u_i > 0$ with $0 < (d/dx_{i})(P|_{H^{+}_{\{i\}}})(u_i) < 1$.  Moreover, $\dfrac{\partial P_i}{\partial x_j}(u_i) < 0$ $(j \ne i)$.}

\vskip 2mm
\noindent \textbf{(H5)} {\em If $x$ is a nontrivial $p$-periodic
point of $P$ and $I \subset \{1,2,3\}$ is such that $x \in \dot{H}^+_I$,
then $\mu_{I,p}(x) < 1$, where $\mu_{I,p}(x)$ is the
\textup{(}necessarily real\textup{)} eigenvalue of the mapping
$D(P|_{H^+ _I})^p(x)$ with the smallest modulus.}

\vskip 2mm
\noindent \textbf{(H6)}  {\em For each nonempty subset $I \subset \{1,2,3\}$
and $x, y \in \dot{H}^+_I$, if $0 \ll_{I} Px \ll_{I} Py$, then $\dfrac{P_{i}x}{P_{i}y} \ge \dfrac{x_i}{y_i}$ for all $i \in
I$ \textup{(}where $P = (P_1, P_2, P_3)$\textup{)}.}

\vskip 2mm
Maps satisfying (H1), (H2) and some weaker form of (H3$'$) (without the last sentence) are called in~\cite{Sm} {\em competitive maps\/}.

\begin{theorem}
\label{thm0}
There exists a compact invariant $S \subset C$ \textup{(}the {\em carrying simplex} for $P$\textup{)} having the following properties:
\begin{enumerate}
\item[{\rm (i)}]
$S$ is homeomorphic to the standard probability simplex $\triangle$ via radial projection $R$.
\item[{\rm (ii)}]
No two points in $S$ are related by the $\ll$ relation.  Moreover,
no two points in $S \cap C^{\circ}$ are related by the $<$ relation.
\item[{\rm (iii)}]
For any $x \in C \setminus \{0\}$ one has $\omega(x) \subset S$.
\item[{\rm (iv)}]
The global attractor $\Gamma$ equals $\{\, {\alpha} x : \alpha \in [0, 1], \ x \in S \,\}$.
\end{enumerate}
\end{theorem}
\begin{proof}
Parts (i) through (iii) are just the corresponding parts in~\cite[Thm.~0]{J-M-W}.  Part (iv) is \cite[Prop.~2.4]{J-M-W}.
\end{proof}

We let $S_I := S \cap H_I^+$, $\dot{S}_I = S \cap \dot{H}_I^+$, $\partial S_I := S \cap \partial H_I^+$ and $S^{\circ} := S \cap C^{\circ}$.  A set $S_I$ is called a {\em $k$-dimensional face\/} of $S$, $k = \card{I}-1$. The union $\partial_k S$ of all $k$-dimensional faces of $S$ is referred to as the {\em
$k$-dimensional skeleton\/} of $S$.

\smallskip
Recall that $V$ is the $3$-dimensional real vector space. $e_i$, $i \in \{1, 2, 3\}$ stands for the $i$-th vector in the standard basis of $V$.

\smallskip
Denote by $\Fix(P)$ the set of fixed points of $P$.  The fixed points $u_{1}$, $u_{2}$ and $u_{3}$ are the {\em axial\/} fixed points, and $x \in \Fix(P) \cap \dot{S}_{\{i, j\}}$ with $i \ne j$ are called {\em planar\/} fixed points.  Observe that planar fixed points need not exist (see, e.g, \cite[Ex.~8.1]{J-M-W}).

\bigskip
It is proved in~\cite[Thm.~A]{J-M-W} that the carrying simplex $S$ is a $C^1$ manifold-with-corners neatly embedded in $C$ if and only if for each ergodic invariant Borel probability measure $\mu$ supported on the boundary $\partial S$ the principal Lyapunov exponent for $\mu$ is less than any of the external Lyapunov exponents for $\mu$.  We will explain shortly what is meant under those terms.

A general definition of Lyapunov exponents (in arbitrary dimension) requires introducing concepts from ergodic theory, such as ergodic invariant measures, Oseledets decomposition, etc., which would reach far beyond the scope of the present paper.  Suffice it to say that a Lyapunov exponent is the logarithmic growth rate of a tangent vector under the derivatives of iterates of the map.

In our case, the definition of Lyapunov exponent almost trivializes.  The boundary $\partial S$ is the union of $S_{\{1, 2\}}$, $S_{\{2, 3\}}$ and $S_{\{1, 3\}}$.  Each of those sets is, by (H2), invariant, and, by Theorem~\ref{thm0}(i), homeomorphic to a compact interval.  Further, their end-points are axial fixed points, therefore $P|_{S_{\{i, j\}}}$, $i \ne j$, is conjugate to an orientation-preserving homeomorphism of a compact interval.  Consequently, ergodic invariant probability measures supported on $\partial S$ are Dirac measures on axial and planar fixed points of $P$.  The Lyapunov exponents for such measures are logarithmic growth rates of vectors in $V$ under positive iterates of the Jacobian matrices of $P$ at the fixed point, that is, the natural logarithms of the moduli of the eigenvalues of those matrices.

\smallskip
Consider first the axial fixed points $u_i$.  For definiteness, let $i = 1$.  The Jacobian matrix $DP(u_1)$ has the form
\begin{equation*}
\begin{pmatrix}
a_{11} & * & *
\\
0 & a_{22} & *
\\
0 & 0 & a_{33}
\end{pmatrix}.
\end{equation*}
As $DP(u_1)$ is non-singular, the diagonal entries are non-zero.  By (H4$'$), $0 < a_{11} < 1$.  We claim that both $a_{22}$ and $a_{33}$ are positive.  Indeed, if, for example, $a_{22} < 0$ then the image under $P$ of the interval with end-points $x$ and $x + {\epsilon} e_2$, with $\epsilon > 0$ sufficiently small, would be a $C^1$ arc tangent at $x$ to a vector whose second coordinate is negative, a contradiction to (H2).  For $\mu = \delta_{u_1}$, the principal Lyapunov exponent is $\ln{a_{11}}$, and external Lyapunov exponents are $\ln{a_{22}}$ and $\ln{a_{33}}$.  We will call $a_{11}$ the {\em principal eigenvalue at $x$\/}, and $a_{22}$ (resp.\ $a_{33}$) the {\em external eigenvalue at $x$\/} {\em corresponding to the second species\/} (resp.\ {\em corresponding to the third species\/}), cf.~\cite[Def.~2.2]{J-M-W}.  Observe that the principal eigenvalue at $u_1$ (in our sense) equals the reciprocal of the principal eigenvalue of the positive $1 \times 1$ matrix $(DP^{-1}(u_1))_{\{1\}}$ (in the sense of the Perron--Frobenius theory, see, e.g., \cite[Thm.~1.1]{Sen}).

Let $x$ be a planar fixed point.  For definiteness, assume that $x \in \dot{S}_{\{1, 2\}}$. The Jacobian matrix $DP(x)$ has the form
\begin{equation*}
\begin{pmatrix}
b_{11} & b_{12} & *
\\
b_{21} & b_{22} & *
\\
0 & 0 & b_{33}
\end{pmatrix}.
\end{equation*}
In a manner similar to that above we show that $b_{33}$ is positive.  Its natural logarithm is the external Lyapunov exponent for $\delta_{x}$.  Observe that
$\bigl(\begin{smallmatrix}
b_{11} & b_{12}
\\
b_{21} & b_{22}
\end{smallmatrix}\bigr)$
is $(DP(x))_{\{1, 2\}}$.  By (H3$'$), its inverse $(DP^{-1}(x))_{\{1, 2\}}$ has all entries positive.  By the Perron--Frobenius theorem, it has two real eigenvalues, the (positive) one, $\lambda_1$, with larger modulus corresponding to an eigenvector with both coordinates positive, and the other one, $\lambda_2$, with smaller modulus, corresponding to an eigenvector with coordinates of opposite signs.

We claim that $\lambda_2$ is positive, too.  By (H5), $\lambda_1 > 1$, so there exists a (unique) local most unstable manifold for $P^{-1}|_{H_{\{1, 2\}}}$ at $x$.  The most unstable manifold is ordered by $\ll_{\{1, 2\}}$, so it intersects $S$ only at $x$.  $S$ is a locally invariant one-dimensional (topological) manifold, so, by theorems on locally invariant manifolds (see, e.g., \cite[Thm.~5.2 and Cor.~5.4 and~5.5]{H-P-S}), it is tangent at $x$ to an eigenvector pertaining to $\lambda_2$.  If $\lambda_2$ were negative, $P|_{S_{\{1, 2\}}}$ would be orientation-reversing, a contradiction.

The internal Lyapunov exponents for $\delta_x$ are the natural logarithms of $1/\lambda_1$ and $1/\lambda_2$.  The principal Lyapunov exponent for $\delta_x$ is the natural logarithm of $1/\lambda_1$.  We will call $b_{33}$ the {\em external eigenvalue at $x$\/}, and $1/\lambda_1$ and $1/\lambda_2$ the {\em internal eigenvalues at $x$\/}. Further, we will refer to $1/\lambda_1$ as the {\em principal eigenvalue at $x$\/} (again, the principal eigenvalue at $x$ in our sense equals the reciprocal of the principal eigenvalue of the positive $2 \times 2$ matrix $(DP^{-1}(x))_{\{1, 2\}}$ in the sense of the Perron--Frobenius theory).

\section{Main Theorem}
\label{sect:main}

Recall that $R$ stands for the radial projection of the carrying simplex $S$ onto the probability simplex $\Delta$ (see Theorem~\ref{thm0}(i)).

\begin{definition}
\label{def:convex}
$S$ is {\em convex\/} if the set $\{\, {\alpha} x : \alpha \in [0, 1], \ x \in S \,\}$, that is, the global attractor $\Gamma$, is convex.
\end{definition}

\begin{theorem}
\label{thm:main}
If $S$ is convex then it is a $C^1$ submanifold-with-corners neatly embedded in $C$, diffeomorphic to $\triangle$ via radial projection.
\end{theorem}

The reader not well-versed in differential topology can imagine a `$C^1$ submanifold-with-corners diffeomorphic to $\triangle$ via radial projection' in the following way: the inverse mapping $(R|_S)^{-1}$ can be written as
\begin{equation}
\label{eq-neat}
(R|_S)^{-1}(y) = \rho(y) y, \quad y \in \triangle,
\end{equation}
where $\rho \colon \triangle \to (0, \infty)$ is a $C^1$ function (recall that $C^1$ means that there is an open neighbourhood $\tilde{U}$ of $\triangle$ in $\{\, x \in V: \sum_{i = 1}^{3} x_i = 1 \,\}$ and a $C^1$ function $\tilde{\rho} \colon \tilde{U} \to (0, \infty)$ such that $\tilde{\rho}|_{\triangle} = \rho$).  The meaning of `neat embedding' is as follows.  For $x \in \partial S$ let $I \subset \{1, 2, 3\}$ be such that $x \in \dot{S}_I$.  Then, $S$ is neatly embedded if for any $x \in \partial S$ the tangent space $\mathcal{T}_{x} S$ of $S$ at $x$ is transverse to $V_I$, meaning that $\mathcal{T}_{x} S + V_I = V$ (informally speaking, $S$ is neatly embedded in $C$ if at any $x \in \partial S$ there is as little tangency to the corresponding face of $C$ as possible).  The standard probability simplex $\triangle$ is neatly embedded in $C$, and the existence of a $C^1$ function as in~\eqref{eq-neat} implies that $S$ is neatly embedded in $C$.

We prove Theorem~\ref{thm:main} by induction on the dimension of the skeleton.  For $0$-dimensional faces the statement is obvious.  Regarding the two-dimensional faces we have the following.
\begin{proposition}
\label{prop-dim2}
If $S$ is convex then for any $i \ne j$ the one-dimensional face $S_{\{i, j\}}$ is a one-dimensional $C^1$ submanifold-with-corners, neatly embedded in $S_{\{i, j\}}$.  Moreover, there is an invariant Whitney sum decomposition
\begin{equation*}
S_{\{i, j\}} \times V_{\{i, j\}} = \mathcal{T} S_{\{i, j\}} \oplus \mathcal{R}_{\{i, j\}}
\end{equation*}
\textup{(}$\mathcal{T} S_{\{i, j\}}$ stands for the tangent bundle of $S_{\{i, j\}}$\textup{)} with the following properties:
\begin{enumerate}
\item[\textup{(i)}]
the fiber $\mathcal{R}_{\{i, j\}}(x)$ of $\mathcal{R}_{\{i, j\}}$ over $x$ can be written as $\spanned\{r(x)\}$, where $r \colon S_{\{i, j\}} \to V_{\{i, j\}}$ is continuous, with $\norm{r(x)} = 1$ for all $x \in S_{\{i, j\}}$, and
\begin{itemize}
\item
$r(x) \in \dot{K}_{\{i, j\}}$ if $x \in \dot{S}_{\{i, j\}}$,
\item
$r(u_i) = e_{i}$ and $r(u_j) = e_{j}$;
\end{itemize}
\item[\textup{(ii)}]
the tangent space $\mathcal{T}_{x} S_{\{i, j\}}$ of $S_{\{i, j\}}$ at $x$ can be written as $\spanned\{w(x)\}$, where $w \colon S_{\{i, j\}} \to V_{\{i, j\}}$ is continuous, with $\norm{w(x)} = 1$ and $w(x) \not\in K_{\{i, j\}}$ for all $x \in S_{\{i, j\}}$;
\item[\textup{(iii)}]
at $x \in \Fix(P) \cap \dot{S}_{\{i, j\}}$, $r(x)$ is the normalized eigenvector in $\dot{K}_{\{i, j\}}$ pertaining to the principal eigenvalue, and $w(x)$ is a normalized eigenvector pertaining to the other internal eigenvalue;
\item[\textup{(iv)}]
there are $C > 0$ and $\nu > 0$ such that
\begin{equation}
\label{eq:exp-sep}
\frac{\norm{DP^n(x) r(x)}}{\norm{DP^n(x) v(x)}} \le C e^{{-\nu} n}
\end{equation}
for all $x \in S_{\{i, j\}}$ and all $n \in \NN$.
\end{enumerate}
\end{proposition}
The property described in~\eqref{eq:exp-sep} is called {\em exponential separation\/}.
\begin{proof}[Indication of proof]
In the light of \cite[Thms~3.1 and~5.1]{J-M-W} it suffices to prove that the convexity of $S$ implies that at each axial fixed point $u_i$ the principal eigenvalue is smaller than the external eigenvalues.  The proof is a simplified (two-dimensional) version of the proof of Proposition~\ref{prop-dim3}, so we do not give it now.
\end{proof}

\begin{proposition}
\label{prop-dim3}
Assume that $S$ is convex.  Then at any planar fixed point $x \in \dot{S}_{\{i, j\}}$ the principal eigenvalue is smaller than both the internal eigenvalues.
\end{proposition}

The idea of the proof is, in~short, to show that the convexity of $S$ entails that there exists an eigenvector of $DP(x)$ not in $V_{\{i, j\}} \cup K$.  Then it turns out that the existence of such an eigenvector implies the statement.

For $x \in S$ we define
\begin{equation*}
\mathcal{C}_1(x) := \{\, v \in V: \exists\ (x^{(n)})_{n = 1}^{\infty} \subset S \setminus \{x\}, x^{(n)} \to x, \frac{x^{(n)} - x}{\norm{x^{(n)} - x}} \to v \,\}.
\end{equation*}
$\mathcal{C}(x) : = [0, \infty) \cdot \mathcal{C}_1(x)$ is called the {\em tangent cone\/} of $S$ at $x$.  $\mathcal{C}(x)$ is a nontrivial (that is, not containing only $0$) closed subset of $V$.  Further, $DP(x) \mathcal{C}(x) = \mathcal{C}(P(x))$ for any $x \in S$.

\medskip
For definiteness, put $i = 1$, $j = 2$, and denote $I = \{1, 2\}$.

Take $x \in \Fix(P) \cap \dot{S}_I$, and write $r$ for $r(x)$, and $w$ for $w(x)$ (in other words, $r$ is the normalized eigenvector in $\dot{K}_{I}$ corresponding to the principal eigenvalue at $x$, and $w$ is a normalized eigenvector corresponding to the other internal eigenvalue at $x$).  Let  $\mathcal{C}_1$ stand for $\mathcal{C}_1(x)$.  We have that $\mathcal{C}_1 \cap \mathcal{T}_{x} S_I = \{w, - w\}$ is a proper subset of $\mathcal{C}_1$.

We decompose $z \in \mathcal{C}_1$ as
\begin{equation}
\label{decomp_z}
z = \alpha(z) e_3 - \beta(z) r + {\gamma(z)} w,
\end{equation}
with real $\alpha(z)$, $\beta(z)$ and $\gamma(z)$.  Since the third coordinates of $r$ and $w$ are zero and $\norm{z} = 1$, we have $0 \le \alpha(z) \le 1$.

\begin{lemma}
\label{lm-1}
$\beta(z) \ge 0$ for any $z \in \mathcal{C}_1$.
\end{lemma}
\begin{proof}
Suppose that there is $z \in \mathcal{C}_1$ such that $\beta(z)$ in~\eqref{decomp_z} is negative.  Then $\alpha(z) e_3 - \beta(z) r \in K^{\circ}$.

Assume that $\gamma(z) \ne 0$.  Since $z \in \mathcal{C}_1$, there is a sequence $x^{(n)} \in S \setminus \{x\}$ converging to $x$ and such that for each $\epsilon > 0$
\begin{equation}
\label{eq-1}
\left\lVert \frac{x^{(n)} - x}{\norm{x^{(n)} - x}} - (\alpha(z) e_3 - \beta(z) r + {\gamma(z)} w) \right\rVert < \epsilon
\end{equation}
for $n$ sufficiently large.  As $\dot{S}_I$ is a $C^1$ one-dimensional manifold, for any (sufficiently large) $n$ there exists $\tilde{x}^{(n)} \in \dot{S}_I$ such that $\norm{\tilde{x}^{(n)} - x} = \abs{\gamma(z)} \, \norm{x^{(n)} - x}$.  Further, as $\dot{S}_I$ is tangent at $x$ to $\gamma(z) w$, for each $\epsilon > 0$ there holds
\begin{equation*}
\left\lVert \frac{\tilde{x}^{(n)} - x}{\norm{\tilde{x}^{(n)} - x}} - \frac{{\gamma(z)} w}{\abs{\gamma(z)}} \right\rVert < \frac{\epsilon}{\abs{\gamma(z)}},
\end{equation*}
consequently
\begin{equation}
\label{eq-2}
\left\lVert \frac{\tilde{x}^{(n)} - x}{\norm{x^{(n)} - x}} - {\gamma(z)} w \right\rVert < \epsilon
\end{equation}
for $n$ sufficiently large.  Putting together \eqref{eq-1} and~\eqref{eq-2} we see that
\begin{equation*}
\left\lVert \frac{x^{(n)} - \tilde{x}^{(n)}}{\norm{x^{(n)} - x}} - (\alpha(z) e_3 - \beta(z) r) \right\rVert < 2 \epsilon
\end{equation*}
for $n$ sufficiently large.  Take now $\epsilon > 0$ so small that vectors within $2\epsilon$ of $\alpha(z) e - \beta(z) r$ belong to $K^{\circ}$.  Therefore, for some $n$, $\tilde{x}^{(n)} \ll x^{(n)}$, which is impossible.  The case $\gamma(z) = 0$ is considered in a similar (but simpler) way.
\end{proof}

Let $a$ stand for the eigenvalue of $DP^{-2}(x)$ pertaining to $r$, $DP^{-2}(x) r = a r$.  In other words, $a$ equals the reciprocal of the square of the principal eigenvalue at $x$.

\begin{lemma}
\label{lm-3}
$\beta(z)/\alpha(z)$ is positive and bounded away from zero, uniformly in $z \in \mathcal{C}_1 \setminus \mathcal{T}_x S_{I}$.
\end{lemma}
\begin{proof}
We write
\begin{equation*}
DP^{-2}(x) e_3 = b e_3 + c r + d w,
\end{equation*}
where $b$, $c$ and $d$ are reals.

It follows from (H3$'$) that we have $DP^{-2}(x) e_3 \gg 0$.  Consequently, $b > 0$.  As $r \in \dot{K}_I$, there holds $c > 0$.

Take $z \in \mathcal{C}_1 \setminus \mathcal{T}_x S_{I}$.  It follows from~\eqref{decomp_z} that
\begin{equation*}
DP^{-2}(x) z = {\alpha(z)} b e_3 + ({\alpha(z)} c - {\beta(z)} a) r  + {\eta(z)} w,
\end{equation*}
where $\eta(z) \in \RR$.  Applying Lemma~\ref{lm-1} to $DP^{-2}(x) z/\norm{DP^{-2}(x) z} \in \mathcal{C}_1$ we obtain
\begin{equation*}
\frac{\beta(z)}{\alpha(z)} \ge \frac{c}{a}.
\end{equation*}
\end{proof}

\begin{lemma}
\label{lm-2}
$\alpha(z)/\beta(z)$ is positive and bounded away from zero, uniformly in $z \in \mathcal{C}_1 \setminus \mathcal{T}_x S_{I}$.
\end{lemma}
\begin{proof}
Take an interval, $B$, with end-points $y^{(1)}$, $y^{(2)}$, contained in $\dot{\triangle}_{I}$, such that $R(x) = \tfrac{1}{2} y^{(1)} + \tfrac{1}{2} y^{(2)}$.  Consider the plane $L$ passing through $u_3 = (R|_{S})^{-1}(0, 0, 1)$, $x^{(1)} = (R|_{S})^{-1} (y^{(1)})$, $x^{(2)} = (R|_{S})^{-1} (y^{(2)})$.  A vector $p$ normal to $L$ can be chosen to have all coordinates positive, so the intersection $L \cap C$ divides $C$ into two sets, a bounded one, $L^{-}$, containing the origin, and an unbounded one, $L^{+}$.  By the convexity assumption, the image $(R|_{S})^{-1}(\conv\{(0,0,1), y^{(1)}, y^{(2)}\})$ is contained in $L \cup L^{+}$.  Observe that the image of $\{\, t_{0}(0, 0, 1) + t_{1}y^{(1)} + t_{2}y^{(2)}: t_0 + t_1 + t_{2} = 1, \ t_0 \ge 0, \ t_1, t_{2} > 0 \,\}$ under $(R|_{S})^{-1}$ is a neighbourhood of $x$ in the relative topology of $S$.

The above construction can be repeated when we replace $B$ by its image under the homothety with centre $R(x)$ and ratio $\epsilon \in (0, 1]$.  Let $\epsilon \to 0^{+}$.  Then the planes $L$ converge to the plane $\tilde{L}$ containing $x + \mathcal{T}_{x} S_{I}$ and passing through $u_3$.  Any $z \in \mathcal{C}_1$, considered a bound vector with initial point at $x$, has its terminal point in $\tilde{L} \cup \tilde{L}^{+}$.

A normal vector $\tilde{p}$ to $\tilde{L}$ can be chosen to belong to $K^{\circ}$.  From the previous paragraph it follows that for any $z \in \mathcal{C}_1$ there holds $\langle z, \tilde{p} \rangle \ge 0$, consequently, taking~\eqref{decomp_z} into account we obtain
\begin{equation*}
\alpha(z) \langle e_3, \tilde{p} \rangle - \beta(z) \langle r, \tilde{p} \rangle + \gamma(z) \langle w, \tilde{p} \rangle \ge 0.
\end{equation*}
As $\langle e_3, \tilde{p} \rangle > 0$, $\langle r, \tilde{p} \rangle > 0$ and $\langle w, \tilde{p} \rangle = 0$, we have that
\begin{equation*}
\frac{\alpha(z)}{\beta(z)} \ge \frac{\langle r, \tilde{p} \rangle}{\langle e_3, \tilde{p} \rangle}.
\end{equation*}
\end{proof}
It follows from Lemma~\ref{lm-3} and the proof of Lemma~\ref{lm-2} that $z \in \mathcal{C}_1$ belongs to $\mathcal{T}_x S_{I}$ (that is, $z = \pm w$) if and only if $\alpha(z) = \beta(z) = 0$.  We will not use that property in the sequel.

\begin{proof}[Proof of Proposition~\ref{prop-dim3}]

By Proposition~\ref{prop-dim2}, the principal eigenvalue at $x$ is smaller than the other internal eigenvalue.  It suffices then to show that the principal eigenvalue is smaller than the external eigenvalue.

Suppose to the contrary that the external eigenvalue is less than or equal to the principal eigenvalue.  Let $E$ stand for the two-dimensional $DP(x)$-invariant subspace of $V$ corresponding to the least eigenvalue (when the external eigenvalue equals the principal eigenvalue) or to the least and second least eigenvalues (when the external eigenvalue is smaller than the principal eigenvalue).  We have $V = E \oplus \spanned\{w\}$.

We claim that the set of vectors $z \in \mathcal{C}_1 \cap E$ is nonempty.  To prove the claim, notice that $E$ contains a vector of the form $e_3 + {\gamma}_1 w + {\gamma}_2 r$.  Then $\{\tilde{e}, r\}$, where $\tilde{e} = e_3 + {\gamma}_1 w$, is a basis of $E$.
Take the set $R := x + [-{\epsilon} r, {\epsilon} r] + [0, \eta \tilde{e}]$, where $\epsilon > 0$ and $\eta > 0$ are small.  Notice that  $x - {\epsilon} r \in \Gamma$ and $x + {\epsilon} r \in C \setminus \Gamma$, hence, by taking $\eta > 0$ smaller if necessary, we have $x - {\epsilon} r + {\delta} \tilde{e} \in \Gamma$ and $x + {\epsilon} r + {\delta} \tilde{e} \in C \setminus \Gamma$ for all $\delta \in [0, \eta]$.  Therefore for each $\delta \in [0, \eta]$ there is $x(\delta) \in S \cap (x + {\delta} \tilde{e} + [-{\epsilon} r, {\epsilon} r])$.  Since no two points in $S \cap C^{\circ}$ can be in the $<$ relation (Theorem~\ref{thm0}(ii)), such $x(\delta)$ is unique.  As a consequence, the set $S \cap R$ can be written as
\begin{equation*}
\{\, x + {\delta} \tilde{e} - j(\delta) r : \delta \in [0, \eta] \,\},
\end{equation*}
where $j \colon [0, \eta] \to \RR$ with $j(0) = 0$.  Since $j$ is the coordinate of the inverse of the restriction to the compact set $S \cap R$ of the affine projection on $x + \spanned\{\tilde{e}\}$ along $r$, the function $j$ is continuous.  The vectors in $\mathcal{C}_1 \cap E$ correspond to limits of right-hand difference quotients of $j$ at zero.  Lemma~\ref{lm-2} implies that those limits are bounded from above.  The claim thus follows.

Further, by Lemma~\ref{lm-3}, the limits of the right-hand difference quotients of $j$ at zero are positive.  Since $j$ is continuous, it follows that $\mathcal{C}_1 \cap E$ is homeomorphic to a compact interval or to a singleton.  As the mapping defined on $\mathcal{C}_1$ by
\begin{equation*}
v \mapsto \frac{DP(x) v}{\norm{DP(x) v}}, \quad v \in \mathcal{C}_1,
\end{equation*}
is continuous and takes $\mathcal{C}_1 \cap E$ into itself, it has a fixed point, which corresponds to an eigenvector of $DP(x)$ contained in $\mathcal{C}_1 \cap E$.  Denote such a eigenvector by $z$.

\smallskip
If the eigenvalue corresponding to $z$ equals the principal eigenvalue at $x$, each nonzero vector in $E$ is an eigenvector.   Observe that for any $\epsilon \in \RR$ the subspace $E$ contains the vector $r + {\epsilon} \tilde{e} = r + {\epsilon} e_3 + {\epsilon} {\gamma_1} w$.  For $\epsilon > 0$ sufficiently small, such a vector, call it $\tilde{v}$, belongs to $K^{\circ}$.  On the other hand, $z$ has third coordinate positive and does not belong to $K$.  Hence the interval joining $z$ to $\tilde{v}$ intersects $\partial K$ at some $\hat{v}$ with third coordinate positive.  By (H3$'$), $DP^{-2}(x) \hat{v} \in K^{\circ}$, so $\hat{v}$ is not an eigenvector.

\smallskip
There remains the case that the eigenvalue corresponding to $z$ is smaller than the principal eigenvalue at $x$.  Then it is the smallest eigenvalue of $DP(x)$, of algebraic multiplicity one.  Its reciprocal is the largest eigenvalue of the matrix $DP^{-1}(x)$ with non-negative entries.  But this eigenvalue corresponds to an eigenvector outside $K$, which is in contradiction to the Perron--Frobenius theorem (see, e.g., \cite[Thm.~2.1.1]{B-P})
\end{proof}

Since for each ergodic invariant measure supported on $\partial S$ its principal Lyapunov exponent is smaller than its external Lyapunov exponent(s), an application of~\cite[Thm.~A]{J-M-W} concludes the proof of~Theorem \ref{thm:main}.

\section{Concluding remarks}
\label{sect:conclusions}
Analysis of the proofs above shows that it suffices to assume that the carrying simplex is convex near its boundary $\partial S$.

\medskip
One can say that at $x \in \dot{S}_{I}$ with nonempty $I \ne \{1, 2, 3\}$ the carrying simplex $S$ {\em is not tangent to $\partial C$\/} if there holds $(\mathcal{C}(x) - \mathcal{C}(x)) + V_{I} = V$.  Observe that we have proved in~fact the equivalence of the following properties:
\begin{itemize}
\item
$S$ is not tangent to $\partial C$ at any $x \in \Fix(P) \cap \partial S$.
\item
$S$ is a $C^1$ submanifold-with-corners neatly embedded in $C$.
\end{itemize}

\section*{Acknowledgements}
The author is grateful to Stephen Baigent for calling his attention to some bibliographical items, and to the anonymous referees for their remarks.

\section*{Funding}
This work was supported by project 0401/0228/16.

\end{document}